\documentclass[11pt,reqno]{amsart}

\usepackage{amsfonts,amsmath,amssymb,amsthm}
\usepackage{mathtools}
\usepackage{anysize}
\marginsize{2cm}{2cm}{1.5cm}{1.5cm}
\usepackage[utf8]{inputenc}
\usepackage[english]{babel}
\usepackage{cmap}
\usepackage{xcolor}
\usepackage{hyperref}

\hypersetup{
  colorlinks = true,
  urlcolor = blue,
  linkcolor = blue,
  citecolor = red,
  pdftitle = {A colorful quantitative Helly theorem for volume},
  pdfauthor = {Grigory Ivanov},
  pdfsubject = {Combinatorial convexity}
}

\newtheoremstyle{theorem}
  {\topsep}{\topsep}{\itshape}{}
  {\bfseries}{.}{5pt plus 1pt minus 1pt}{}
\theoremstyle{theorem}
\newtheorem{thm}{Theorem}[section]
\newtheorem{prp}{Proposition}[section]
\newtheorem{lem}{Lemma}[section]
\newtheorem{conj}{Conjecture}[section]
\theoremstyle{definition}
\newtheorem{dfn}{Definition}[section]

\newcommand{\Href}[2]{\hyperref[#2]{#1~\ref{#2}}}
\newcommand{\st}{:\;}
\newcommand{\enorm}[1]{\left|#1\right|}
\DeclareMathOperator{\tr}{\mathrm{trace}}
\newcommand{\conv}{\mathrm{conv}}
\newcommand{\iprod}[2]{\left\langle#1,#2\right\rangle}
\DeclareMathOperator{\inter}{int}
\providecommand{\parenth}[1]{\left(#1\right)}
\providecommand{\braces}[1]{\left\{#1\right\}}
\providecommand{\brackets}[1]{\left[#1\right]}
\def\R{{\mathbb R}}
\def\Sph{{\mathbb S}}
\def\Lin{\mathop{\rm Lin}}
\def\diag{\mathop{\rm diag}}

\newcommand{\id}{\mathrm{Id}}
\newcommand{\poscone}[1]{\mathrm{Pos}\ \! {#1}}
\newcommand{\vol}[1]{\operatorname{vol}\nolimits_{#1}}
\providecommand{\abs}[1]{\lvert#1\rvert}
\providecommand{\card}[1]{\lvert#1\rvert}
\newcommand{\ball}[1]{\mathbf{B}^{#1}}
\newcommand{\transpose}[1]{{#1}^{T}}

\numberwithin{equation}{section}
\providecommand{\Bd}{\ball{d}}
\providecommand{\Sd}{\Sph^{d-1}}

\title[A colorful quantitative Helly theorem for volume]
{A colorful quantitative Helly theorem for volume}
\author{Grigory Ivanov}
\email{grimivanov@gmail.com}
\date{}

\begin{document}

\begin{abstract}
We prove a colorful quantitative Helly theorem for volume with the
optimal number $2d$ of colors. If every rainbow intersection from
$2d$ finite families of convex sets in $\R^d$ has volume at least
one, then the intersection of one family has volume at least
$d^{-O(d^2)}$. We also prove a colorful quantitative Steinitz theorem
for origin-centered ellipsoids of different shapes. The proof uses
a common normalization of positive operators and a lift that
produces two rainbow bases with large determinants.
\end{abstract}

\maketitle

\section{Introduction}

The classical Helly theorem \cite{helly1923mengen} states that a
finite family of convex sets in $\R^d$ has non-empty intersection
provided that every subfamily of at most $d+1$ members has
non-empty intersection. B\'ar\'any, Katchalski, and Pach
\cite{barany1982quantitative,barany1984helly} initiated the
quantitative study of this theorem, replacing non-emptiness by
lower bounds on volume or diameter. For volume, they proved that
if every subfamily of at most $2d$ members has intersection of
volume at least one, then the whole family has intersection of
volume at least $d^{-O(d^2)}$. The number $2d$ is optimal.

Nasz\'odi \cite{naszodi2016proof} resolved their conjecture by
improving the volume bound to $d^{-O(d)}$. This has the correct
order up to an absolute constant in the exponent. In the
equivalent selection formulation, one seeks $2d$ sets whose
intersection has volume at most $v_d$ times the volume of the
original intersection. Nasz\'odi proved
$v_d \le e^{d+1}d^{2d+1/2}$, while $v_d \ge d^{d/2}$ is necessary.
The best known upper bound, $v_d \le (Cd)^{3d/2}$ with an absolute
constant $C$, was proved by Brazitikos
\cite{brazitikos2017brascamp} and recovered by
Almendra-Hern\'andez, Ambrus, and Kendall
\cite{almendra2022quantitative}. The latter proof develops the
sparse approximation approach of Ivanov and Nasz\'odi
\cite{ivanov2022quantitative}, who showed that the selected
intersection can be contained in a homothet of the original
intersection with absolute homothety ratio at most $(2d)^3$.
For a recent account of quantitative Helly and Steinitz theorems
and related open problems, see Nasz\'odi
\cite{naszodi2025quantitative}.

The colorful version of the problem has proved more difficult.
De Loera, La Haye, Rolnick, and Sober\'on
\cite{de2017quantitative} proved a colorful volume theorem with
arbitrarily small volume loss, allowing the number of colors to
depend on the dimension and the loss. Rolnick and Sober\'on
\cite{rolnick2017quantitative} used floating bodies to prove
quantitative colorful Helly theorems for a wider class of
functions. Using parameter spaces of ellipsoids, Sarkar, Xue,
and Sober\'on \cite{sarkar2021quantitative} obtained a volume
theorem with $d(d+3)/2$ colors.
Dam\'asdi, F\"oldv\'ari, and Nasz\'odi
\cite{damasdi2021colorful} obtained a theorem with $3d$ colors:
if every rainbow intersection of $2d$ sets has volume at least
one, then one monochromatic intersection has volume at least
$d^{-O(d^2)}$. Recently, Ivanov and Nasz\'odi
\cite{Ivanov2026Hellynumbers} proved the colorful quantitative
Helly theorem for diameter with the optimal number $2d$ of
colors. For volume in the plane, B\'ar\'any, Miraftab, and
Theocharous 
recently
proved a colorful Helly theorem with the optimal number of four
colors. The main result of this paper establishes the volume
theorem in arbitrary dimension.

For a positive integer $n$, write
$\brackets{n} := \braces{1, \dots, n}$.
For a set $X \subset \R^d$, we write $\conv X$ for its convex hull
and $\vol{d}X$ for its $d$-dimensional volume. A \emph{rainbow
selection} from several families contains one member from each
family. A \emph{rainbow set} contains at most one member from each
family. Members with the same geometric position but different
color labels are treated as distinct.

\begin{thm}[Colorful quantitative Helly theorem for volume]
\label{thm:colorful_volume_helly}
For every $d \ge 1$ there is $\gamma_d > 0$ with the following
property. Let $\mathcal F_1, \dots, \mathcal F_{2d}$ be finite
non-empty families of convex sets in $\R^d$. Suppose that
\[
 \vol{d}\parenth{\bigcap_{i \in \brackets{2d}}F_i} \ge 1
 \qquad\text{whenever } F_i \in \mathcal F_i
 \text{ for every } i \in \brackets{2d}.
\]
Then, for some $i \in \brackets{2d}$,
\[
 \vol{d}\parenth{\bigcap_{F \in \mathcal F_i}F} \ge \gamma_d.
\]
One may take $\gamma_1 = 1$ and
$\gamma_d = d^{-O(d^2)}$ for $d \ge 2$.
The number $2d$ of colors is optimal.
\end{thm}

The necessity of $2d$ follows by taking identical colors
consisting of the $2d$ facet halfspaces of a cube of arbitrarily
small volume. Every selection of fewer than $2d$ distinct facets
leaves an intersection of infinite volume.

Our approach follows an idea already used by B\'ar\'any,
Katchalski, and Pach \cite{barany1982quantitative}: pass to the
dual problem and use a quantitative Steinitz theorem. In this
problem, a convex hull containing the Euclidean unit ball is
replaced by the convex hull of at most $2d$ of its points,
which must still contain a controlled multiple of the ball.
Ivanov and Nasz\'odi \cite{ivanov2024steinitz} obtained the first
polynomial bound for this multiple. The related passage between
Steinitz and Helly problems by polarity is discussed in
\cite{ivanov_QST_polarity_2025}. For balls with a common center,
a colorful version with $2d$ colors was proved by De Loera,
La Haye, Rolnick, and Sober\'on
\cite[Lemma~2.6]{de2017quantitative}.

In \cite{ivanov2026colorfulsteinitztheoremdifferent}, the author
introduced a lifting to a $2d$-dimensional space for the colorful
problem. It led to a quantitative Steinitz theorem allowing the
balls contained in the color classes to have different centers.
For the volume problem, the corresponding ellipsoids may also
have different shapes, and no common affine map makes all of
them balls. Our main step is the following colorful Steinitz
theorem for ellipsoids with a common center.

Let $\Bd$ and $\Sd$ denote the Euclidean unit ball and unit sphere
in $\R^d$, and put $\kappa_d := \vol{d}\Bd$.

\begin{thm}[Centered ellipsoids in rainbow convex hulls]
\label{thm:ellipsoid_steinitz}
Let $X_1, \dots, X_{2d} \subset \R^d$. Suppose that for every
$i \in \brackets{2d}$ there is an invertible operator $A_i$ such that
\[
 A_i\Bd \subset \conv X_i,
 \qquad \abs{\det A_i} \ge 1.
\]
Then there are a rainbow set $R$ and an invertible operator $A$ with
\[
 A\Bd \subset \conv R,
 \qquad \abs{\det A} \ge c_d,
\]
where $c_d = d^{-O(d^2)}$.
\end{thm}

To prove this theorem, we choose a common extremal position for
the second-moment operators of the colors. A lifting to
$\R^d\times\R^d$ then produces two rainbow bases with large
determinants. We use one basis and a positive representation
in the other group to obtain an ellipsoid in the rainbow hull.
When the monochromatic intersections in the Helly problem have
John ellipsoids with a common center, polarity completes the
argument. To handle different centers, we retain the heights
of the defining halfspaces as well as their normals. Independent
translations enter through these heights, and an affine identity
for balanced colors controls their contribution.

We prove the Helly theorem in its equivalent selection form:
if every monochromatic intersection has volume at most one,
then some rainbow intersection has volume at most $C_d$, where
$C_d = \gamma_d^{-1}$. In dimension one, the smaller of the two
cross differences between the endpoint constraints of the colors
is at most the average of their monochromatic interval lengths;
a negative difference gives an empty intersection. Hence
$\gamma_1 = 1$. We assume $d \ge 2$ throughout the rest of the proof.

\Href{Section}{sec:strategy} explains the proof of
\Href{Theorem}{thm:ellipsoid_steinitz}.
Sections~\ref{sec:normalization} and~\ref{sec:lift} establish the
common normalization and the colorful conic statements.
\Href{Section}{sec:centered_proof} proves the Steinitz theorem
and states the open question for ellipsoids with different centers.
\Href{Section}{sec:affine_proof} contains the polyhedral reduction,
the dual construction, and the proof of the Helly theorem.

\subsection*{Use of AI tools}
ChatGPT was used during the development of the arguments and the
preparation of this manuscript, including the formulation and
checking of intermediate lemmas and revisions of the exposition.
{The author regards the argument in
\Href{Lemma}{lem:cone_density} and the identification of the
mixed-discriminant inequality stated in
\Href{Proposition}{prp:gurvits} as the AI agent's main contributions.}

\section{The idea of the proof}
\label{sec:strategy}

We first prove \Href{Theorem}{thm:ellipsoid_steinitz}, which
allows the shapes of the ellipsoids to vary while keeping their
center at zero.
An independent affine normalization of each color would change
the rainbow convex hulls. We instead choose one volume-preserving
linear map and allow a scalar rescaling within each color.

Before doing so, we pass to finite balanced collections of radial
contractions of the original points.
{This develops the approach of Ivanov and Nasz\'odi
\cite{ivanov2024steinitz}: at the cost of an additional factor in
the final bound, we obtain balanced sets with a controlled number
of vectors.}
Their second-moment operators retain a determinant lower bound from the ellipsoids.

The common normalization gives a second-moment identity to which
a mixed-discriminant inequality applies. After lifting the colors
to $\R^d \times \R^d$, this yields a partition into two groups,
each with a large rainbow basis. We keep the basis from one group
and use the other group to represent the opposite of its
barycenter. The resulting positive balance gives the required
inscribed ellipsoid. The nontrivial selection step is that one
large rainbow basis, together with balance in each color, gives
a covering rainbow basis of controlled determinant in every
direction. This is \Href{Theorem}{thm:conic_completion}.

In the general Helly problem, the halfspaces carry both normal
vectors and heights. Facet areas give balanced normals whose
second moments retain the volume information. Translations change
the individual heights, but not their average within a color.
The affine counterpart of the conic identity then selects a
positive balance with a large determinant and a bounded weighted
height sum. The polar form of the same ellipsoid inclusion gives
the desired volume upper bound.
{This adapts the lifting idea from
\cite{ivanov2026colorfulsteinitztheoremdifferent}. Recording the
heights in an additional coordinate gives a lift to
$\R^{2d+1}$. Normalizing a positive weighted height sum to one
places the resulting point in a $2d$-dimensional affine hyperplane.}

\section{Balanced data and a common extremal position}
\label{sec:normalization}

In this section the ambient dimension is denoted by $n$ and the
number of colors by $m$. This notation will also be used for the
colorful conic statements, which we apply in dimensions $d$ and
$2d$. The letter $N$ denotes an upper bound on the size of a
finite color.

A finite indexed set $Z \subset \R^n$ is \emph{balanced} if
$\sum_{z \in Z}z = 0$. Sums over indexed sets count every label,
including labels attached to coincident vectors.

For real matrices in $\R^{p\times q}$, we use the Hilbert--Schmidt
inner product
\[
 \iprod{A}{B} := \tr{\transpose{A}B}.
\]
For an invertible operator $B$, we write
\[
 B^{-T} := \transpose{(B^{-1})} = \parenth{\transpose{B}}^{-1}
\]
for the transpose of its inverse.
Write $\mathrm{Sym}_n$ for the Euclidean space of real symmetric
operators on $\R^n$, and $\mathrm{Sym}_n^+$ for its cone of
positive definite operators. Thus, $\iprod{A}{\id_n}$ is the trace of $A$.
For vectors, $\iprod{\cdot}{\cdot}$ denotes their usual scalar
product.

We choose the common position so that the second-moment operators
satisfy the hypotheses of Gurvits's mixed-discriminant inequality:
each operator has trace one, and their sum is a scalar multiple
of the identity. When the number of colors equals the dimension,
these are precisely the \emph{doubly stochastic conditions}.
The normalization is an operator-scaling step; see
\cite[Theorem~5.1]{gurvits2006van} for the case of $n$ operators
on $\R^n$. The use of constrained optimality conditions in this
setting goes back to Fritz John \cite{John}, as discussed in
\cite[Section~3.2]{gurvits2006van}. A related variational approach
to second-moment identities appears in
\cite{ivanov2025johnellipsoidsrevolution}.

Given $C_1, \dots, C_m \in \mathrm{Sym}_n^+$, we minimize
\[
 \mathcal F(H) := \sum_{i \in \brackets{m}}\ln\iprod{C_i}{H},
 \qquad H \in \mathrm{Sym}_n^+,\quad \det H = 1.
\]
The identity
$\iprod{C_i}{H} = \iprod{H^{1/2}C_iH^{1/2}}{\id_n}$
will translate its optimality condition into the desired
normalization of the vectors.

\begin{lem}
\label{lem:normalization}
Let $C_1, \dots, C_m$ be positive definite operators on $\R^n$.
There is a positive definite operator $H$ with $\det H = 1$ such that
\begin{equation}
 \sum_{i \in \brackets{m}}\frac{C_i}{\iprod{C_i}{H}}
 = \frac{m}{n}H^{-1}.
 \label{eq:positive_normalization}
\end{equation}
Moreover, for every $i \in \brackets{m}$,
\[
 \iprod{C_i}{H} \ge n\parenth{\det C_i}^{1/n} > 0.
\]
\end{lem}

\begin{proof}
For $H \in \mathrm{Sym}_n^+$, one has
$\iprod{C_i}{H} \ge \lambda_{\min}(C_i)\iprod{H}{\id_n}$.
Thus,
\[
 \mathcal F(H) \ge 
 \sum_{i \in \brackets{m}}\ln\lambda_{\min}(C_i)
 +m\ln\iprod{H}{\id_n}.
\]
On every sublevel set the eigenvalues of $H$ are bounded above.
Their product is one, so they are also bounded away from zero.
The sublevel sets are therefore compact, and a minimizer exists.

At a minimizer, the Lagrange multiplier rule for $\ln\det H = 0$
gives
\[
 \sum_{i \in \brackets{m}}\frac{C_i}{\iprod{C_i}{H}} = \alpha H^{-1}.
\]
Indeed, the differentials in a direction $U \in \mathrm{Sym}_n$
are $\sum_{i \in \brackets{m}}\iprod{C_i}{U}/\iprod{C_i}{H}$ and
$\iprod{H^{-1}}{U}$, respectively.
Pairing the displayed identity with $H$ gives $m = \alpha n$,
which proves \eqref{eq:positive_normalization}.

Finally, apply the arithmetic--geometric mean inequality to
the eigenvalues of $H^{1/2}C_iH^{1/2}$:
\[
 \iprod{C_i}{H}
 \ge n\parenth{\det(H^{1/2}C_iH^{1/2})}^{1/n}
 = n\parenth{\det C_i}^{1/n}.\qedhere
\]
\end{proof}

To apply the lemma to finite balanced colors $Z_i$, put
\[
 C_i := \frac{1}{\card{Z_i}}\sum_{z \in Z_i}z\otimes z,
 \qquad s_i := \sqrt{\iprod{C_i}{H}},
 \qquad Y_i := \braces{H^{1/2}z/s_i:z \in Z_i}.
\]
Here the $C_i$ are assumed positive definite. The map $H^{1/2}$
preserves volume, and the scalar $s_i$ is specific to the color.
The second-moment operators of the new colors are
\[
 B_i = \frac{H^{1/2}C_iH^{1/2}}{\iprod{C_i}{H}},
 \qquad \iprod{B_i}{\id_n} = 1,
 \qquad \sum_{i \in \brackets{m}}B_i = \frac{m}{n}\id_n,
\]
where the last identity follows by multiplying
\eqref{eq:positive_normalization} on both sides by $H^{1/2}$.
Their centroids remain at zero.

\begin{dfn}
\label{dfn:normalized_colors}
A family of non-empty finite indexed colors
$Y_1, \dots, Y_m \subset \R^n$ is \emph{normalized} if, for every
$i \in \brackets{m}$,
\begin{equation}
 \sum_{y \in Y_i}y = 0,
 \qquad B_i := \frac{1}{\card{Y_i}}\sum_{y \in Y_i}y\otimes y,
 \qquad \iprod{B_i}{\id_n} = 1,
 \qquad \sum_{i \in \brackets{m}}B_i = \frac{m}{n}\id_n.
 \label{eq:normalized_colors}
\end{equation}
\end{dfn}

If $\card{Y_i} \le N$, every vector has norm at most $\sqrt N$.
The definition allows individual $B_i$ to be singular.
For $m = n$, the operators form a doubly stochastic family.
For $m = 2n$, their sum is $2\id_n$; the lifting used below replaces
$B_i$ by $\frac{1}{2}\diag(B_i,B_i)$ and produces a doubly
stochastic family of $2n$ operators on $\R^{2n}$.

\section{Colorful conic covering}
\label{sec:lift}

Positive balance has a useful consequence: one rainbow basis
with a large determinant yields a covering rainbow basis in every
direction, with a controlled loss in the determinant.
The geometric quantity here is the volume of the parallelotope
spanned by the basis. Together with an upper bound on the vector
lengths, it controls the coefficients of a positive representation.

\begin{dfn}
\label{dfn:rainbow_bases}
For finite indexed colors $Y_1, \dots, Y_n \subset \R^n$, a
\emph{rainbow basis} is an invertible matrix
$B = [y_1\ \dots\ y_n]$ with $y_i \in Y_i$ for
$i \in \brackets{n}$; its columns are ordered by color.
Write $\mathcal B$ for the family of these indexed bases. For $u \in \R^n$, put
\[
 \poscone{B} := \braces{Ba:a \in [0,\infty)^n},
 \qquad
 \mathcal B(u) := \braces{B \in \mathcal B:u \in \poscone{B}}.
\]
Coincident vectors retain their labels. A vector is
\emph{generic} if it avoids the spans of all selections of at
most $n-1$ vectors from the configuration.
\end{dfn}

In the quantitative estimates, we use the constant
\begin{equation}
 \delta_n := \sqrt{\frac{n!}{n^n}} \ge e^{-n/2},
 \qquad n \ge 1.
 \label{eq:delta}
\end{equation}
The inequality follows from $n! \ge (n/e)^n$.

\begin{thm}
\label{thm:conic_completion}
Let $Y_1, \dots, Y_n \subset \R^n$ be non-empty finite indexed colors
with $\sum_{y \in Y_i}y = 0$ for every $i \in \brackets{n}$.
Suppose that each color has at most $N$ vectors, all of norm at
most $R$, and that some rainbow basis has absolute determinant
at least $D>0$. Then every $u \in \R^n$ has a rainbow basis
$B = [r_1\ \dots\ r_n]$ and a positive representation
\[
 u = \sum_{i \in \brackets{n}}a_i r_i,\qquad a_i\ge0,
\]
such that, for every $i \in \brackets{n}$,
\[
 \abs{\det B} \ge \frac{D}{N^n},
 \qquad
 a_i \le \frac{N^nR^{n-1}}{D} \enorm{u}
.
\]
If the $n$ colors are normalized, one may take $D = \delta_n$
and $R = \sqrt N$, with $\delta_n$ defined in \eqref{eq:delta}.
\end{thm}

We first prove the qualitative covering statement, a conic
analogue of colorful Steinitz, together with the determinant
identity that gives the quantitative bound. The local step
involves one cone of dimension $n-1$ and one balanced set.

\begin{lem}
\label{lem:normal_shifts}
Let $f_1, \dots, f_{n-1} \in \R^n$ be linearly independent, put
$C := \poscone{\braces{f_1, \dots, f_{n-1}}}$, and let $\nu$ be a unit
normal to $H := \Lin C$. Let $Y$ be a non-empty finite indexed
set with $\sum_{y \in Y}y = 0$. For $y \in Y$, put
\[
 B(y) := [f_1\ \dots\ f_{n-1}\ y],
 \qquad C_y := \poscone{\braces{f_1, \dots, f_{n-1},y}}.
\]
For every $z$ in the relative interior of $C$, there is
$\varepsilon > 0$ such that, for every $y \in Y$ and every
$0 < t < \varepsilon$,
\begin{equation}
 z \pm t\nu \in C_y
 \quad\Longleftrightarrow\quad
 \pm\iprod{\nu}{y} > 0,
 \label{eq:conic_transition}
\end{equation}
with matching signs. Moreover,
\[
 \sum_{\substack{y \in Y\\z+t\nu \in C_y}}\abs{\det B(y)}
 = 
 \sum_{\substack{y \in Y\\z-t\nu \in C_y}}\abs{\det B(y)}
 = \frac{1}{2}\sum_{y \in Y}\abs{\det B(y)}.
\]
\end{lem}

\begin{proof}
Write $z = \sum_{i \in \brackets{n-1}}b_i f_i$, with $b_i > 0$ for every
$i \in \brackets{n-1}$. For $y \in Y\setminus H$, put
$s_y := \iprod{\nu}{y}$ and write
$y = \sum_{i \in \brackets{n-1}}c_i(y)f_i+s_y \nu$. Then
\[
 z+t\nu = 
 \sum_{i \in \brackets{n-1}}\parenth{b_i-\frac{t c_i(y)}{s_y}}f_i
 +\frac{t}{s_y}y.
\]
Since $Y$ is finite, all coefficients of the $f_i$ are positive
for $\abs{t} < \varepsilon$, with one $\varepsilon > 0$
independent of $y$. Membership in $C_y$ is therefore determined
by the sign of $t/s_y$. If $y \in H$, then $C_y \subset H$,
so neither nonzero normal shift belongs to $C_y$. This proves
\eqref{eq:conic_transition} simultaneously for all $y \in Y$.

Let $\omega$ be the $(n-1)$-dimensional volume of the
parallelotope spanned by $f_1, \dots, f_{n-1}$. Set
$s_y := \iprod{\nu}{y}$ also for $y \in Y \cap H$, and put
$Y^\pm := \braces{y \in Y:\pm s_y > 0}$.
The base-times-height formula gives
$\abs{\det B(y)} = \omega\abs{s_y}$. Hence
\[
 \sum_{y \in Y^+}\abs{\det B(y)}
 -\sum_{y \in Y^-}\abs{\det B(y)}
 = \omega\sum_{y \in Y}s_y
 = \omega\iprod{\nu}{\sum_{y \in Y}y} = 0.
\]
The two sums are equal, and their sum is
$\sum_{y \in Y}\abs{\det B(y)}$, since the remaining
determinants vanish. Together with \eqref{eq:conic_transition},
this proves the assertion.
\end{proof}

\begin{lem}[Balanced conic covering]
\label{lem:cone_density}
Let $Y_1, \dots, Y_n$ be non-empty finite indexed colors in $\R^n$
with $\sum_{y \in Y_i}y = 0$ for every $i \in \brackets{n}$.
For $u \in \R^n$, put
\[
 f(u) := \sum_{B \in \mathcal B(u)}\abs{\det B}.
\]
Then $f$ takes the same value at all generic vectors. If a rainbow
basis exists, this value is positive and the rainbow positive
cones cover $\R^n$.
\end{lem}

\begin{proof}
If $\mathcal B$ is empty, then $f = 0$. In dimension one the
assertion follows by summing the positive and negative vectors
in the only color. Thus, assume $n\ge2$ and that a rainbow basis
exists. Let $\mathcal H$ be the finite family of distinct
hyperplanes spanned by facets of the cones $\poscone{B}$, with
$B \in \mathcal B$. Away from their union, $f$ is locally
constant.
Any two chambers of this finite hyperplane arrangement
can be joined by a polygonal path crossing one hyperplane at a
time and avoiding their pairwise intersections. It therefore   suffices to compare the values of $f$ on
opposite sides of one $H \in \mathcal H$ near a point $z \in H$
belonging to no other member of $\mathcal H$.
Let $\nu$ be a unit normal to $H$.

If $z$ is on the boundary of a basis cone, it lies in the relative
interior of exactly one of its facets, whose span is $H$.
Otherwise it would belong to two distinct facet spans.
All cones not having $z$ on their boundary have unchanged
membership under sufficiently small normal displacements of $z$.
Thus, the changing bases split into groups according to their
unique indexed facet containing $z$.

Fix one such facet $F = (f_i)_{i \in I}$, where
$I = \brackets{n}\setminus\braces{k}$. Its group consists
of all nonsingular completions $B_F(y)$ by $y \in Y_k$, with
columns ordered by color. Apply
\Href{Lemma}{lem:normal_shifts} to this facet and the missing
color $Y_k$. It shows that the total determinant weight of
this group is the same at $z+t\nu$ and $z-t\nu$ for sufficiently
small $t > 0$. There are finitely many groups, so one $t$
works for all of them. Summing their equalities gives
$f(z+t\nu) = f(z-t\nu)$. Every indexed basis is counted once,
even if several indexed facets generate the same cone.
This proves constancy on the generic set.

Evaluating $f$ at a generic point in the interior of one rainbow
basis cone shows that its common generic value is positive.
Every generic vector is therefore covered. Each basis cone is
closed, so their finite union is closed; since generic vectors
are dense, this union is all of $\R^n$.
\end{proof}

\subsection{Large rainbow bases}

The normalized case of \Href{Theorem}{thm:conic_completion}
uses a mixed discriminant inequality. The
\emph{mixed discriminant} $D(Q_1, \dots, Q_n)$ is the coefficient of
$t_1\dots t_n$ in $\det(\sum_{i \in \brackets{n}}t_iQ_i)$.
In particular, $D(\id_n, \dots, \id_n) = n!$.

\begin{prp}[Gurvits {\cite[Theorem~4.1]{gurvits2006van}}]
\label{prp:gurvits}
Let $Q_1, \dots, Q_n$ be positive semidefinite real symmetric
$n\times n$ matrices with
\[
 \sum_{i \in \brackets{n}}Q_i = \id_n,
 \qquad \iprod{Q_i}{\id_n} = 1
 \quad\text{for every }i \in \brackets{n}.
\]
Then
\[
 D(Q_1, \dots, Q_n) := 
 \frac{\partial^n}{\partial t_1\cdots\partial t_n}
 \det\parenth{\sum_{i \in \brackets{n}}t_iQ_i}
 \ge \frac{n!}{n^n}.
\]
\end{prp}

Theorem~4.1 of the cited paper proves Bapat's conjecture for
doubly stochastic collections of positive semidefinite Hermitian
matrices. The real symmetric case above is included directly.

We next relate this operator inequality to the determinants of
colorful selections. The required identity is a coefficient form
of the Cauchy--Binet formula. For its isotropic probabilistic version, see
Pivovarov \cite[Lemma~3]{pivovarov2010determinants}; the
Cauchy--Binet formula for second-moment operators is also used
in the tight-frame approach of \cite{ivanovframes}.
We give the finite expansion explicitly.

\begin{prp}[A colorful form of the Cauchy--Binet formula]
\label{prp:colored_cauchy_binet}
Let $Y_i = \braces{y_{ij}:j \in \brackets{m_i}} \subset \R^n$
be finite indexed sets, and let $p_{ij} \ge 0$. Put
$Q_i := \sum_{j \in \brackets{m_i}}p_{ij}y_{ij}\otimes y_{ij}$ for
$i \in \brackets{n}$. Then
\begin{equation}
 D(Q_1, \dots, Q_n) = 
 \sum_{j_1 \in \brackets{m_1}}\dots\sum_{j_n \in \brackets{m_n}}
 \parenth{\prod_{i \in \brackets{n}}p_{i j_i}}
 \abs{\det[y_{1j_1}\ \dots\ y_{nj_n}]}^2.
 \label{eq:colored_cauchy_binet}
\end{equation}
\end{prp}

\begin{proof}
Index the columns by
$\mathcal I := \braces{(i,j):i \in \brackets{n},\,
j \in \brackets{m_i}}$, ordered lexicographically.
For $t_i \ge 0$, form the matrix whose $(i,j)$th column is
$\sqrt{t_i p_{ij}}\,y_{ij}$.
Its product with its transpose is $\sum_{i \in \brackets{n}}t_iQ_i$.
If $J \subset \mathcal I$ has $n$ elements, write $Y_J$
for the square matrix with columns $y_{ij}$ indexed by $J$.
The Cauchy--Binet formula \cite[Section~0.8.7]{horn2012matrix} gives
\[
 \det\parenth{\sum_{i \in \brackets{n}}t_iQ_i}
 = \sum_{\substack{J \subset \mathcal I\\\card{J} = n}}
 \abs{\det Y_J}^2\prod_{(i,j) \in J}t_i p_{ij}.
\]
The monomial $t_1\dots t_n$ occurs exactly when $J$ contains
one index from each color. Equating its coefficients gives
\eqref{eq:colored_cauchy_binet}.
\end{proof}
{We now deduce the existence of a rainbow basis with a
large determinant.}
\begin{lem}
\label{lem:large_rainbow_basis}
Let $Y_1, \dots, Y_n \subset \R^n$ be normalized colors, and
write $Y_i = \braces{y_{ij}:j \in \brackets{m_i}}$ for
$i \in \brackets{n}$. Then there are indices $j_i \in \brackets{m_i}$,
$i \in \brackets{n}$, such that
\[
 \abs{\det[y_{1j_1}\ \dots\ y_{nj_n}]} \ge \delta_n,
\]
with $\delta_n$ as in \eqref{eq:delta}.
\end{lem}

\begin{proof}
Let $B_i$ be the second-moment operator of $Y_i$ from
\Href{Definition}{dfn:normalized_colors}.
Use \Href{Proposition}{prp:colored_cauchy_binet} with
$p_{ij} = 1/m_i$, and then apply
\Href{Proposition}{prp:gurvits}. Using \eqref{eq:delta}, we obtain
\[
 \frac{1}{\prod_{i \in \brackets{n}}m_i}
 \sum_{j_1 \in \brackets{m_1}}\dots\sum_{j_n \in \brackets{m_n}}
 \abs{\det[y_{1j_1}\ \dots\ y_{nj_n}]}^2
 = D(B_1, \dots, B_n) \ge \delta_n^2.
\]
The left-hand side is the arithmetic mean of the squared
determinants of all colorful selections, so at least one has
the required size. The determinant estimate uses the second
moments; the balance assumption is needed for the conic
covering to which we apply it.
\end{proof}

\begin{proof}[Proof of \Href{Theorem}{thm:conic_completion}]

Evaluate the common generic value in \Href{Lemma}{lem:cone_density}
at a point in the interior of the given basis cone. That basis
contributes at least $D$, so, for every generic $u$,
\[
 \sum_{B \in \mathcal B(u)}\abs{\det B} \ge D.
\]
There are at most $N^n$ rainbow bases,
so one covering basis has determinant at least $D/N^n$.
Its coefficients are nonnegative. Cramer's rule and Hadamard's
inequality \cite[Section~0.8.3 and Corollary~7.8.3]{horn2012matrix}
give, for every $i \in \brackets{n}$,
\[
 a_i = \abs{(B^{-1}u)_i}
 \le \frac{\enorm{u}\prod_{\substack{j \in \brackets{n}\\j \ne i}}\enorm{r_j}}{\abs{\det B}}
 \le \frac{N^nR^{n-1}\enorm{u}}{D}.
\]

For an arbitrary $u$, approximate it by generic vectors and pass
to a subsequence on which the selected basis is fixed.
Its cone is closed, and its coefficients depend continuously
on $u$. Finally, for normalized colors,
\Href{Lemma}{lem:large_rainbow_basis} supplies
$D = \delta_n$ from \eqref{eq:delta}, and
\eqref{eq:normalized_colors} gives $R = \sqrt N$.
\end{proof}

\subsection{The lift and two disjoint bases}

Our lifting construction now selects $2n$ vectors which split
into two rainbow bases with controlled determinants.
Together with \Href{Theorem}{thm:conic_completion}, this supplies
the selection step for the colorful Steinitz theorem: keep a
large basis from one group and use the other group to represent
the opposite of its barycenter. The next section prepares the
original colors for this argument.

\begin{lem}
\label{lem:two_groups}
Let $Y_1, \dots, Y_{2n} \subset \R^n$ be normalized, with at most
$N$ vectors in each color. The colors admit a partition into
two groups of $n$ colors, each having a rainbow basis with
absolute determinant at least
\begin{equation}
 D_0(n,N) := \delta_{2n}N^{-n/2}.
 \label{eq:basis_scale}
\end{equation}
Here $\delta_{2n}$ is the constant in \eqref{eq:delta} for dimension $2n$.
\end{lem}

\begin{proof}
Lift every color to
\[
 \widehat Y_i := 
 (Y_i\times\braces{0})\cup(\braces{0}\times Y_i)
 \subset \R^{2n}.
\]
The two indexed copies retain their original color.
Their second-moment operator is
$\widehat B_i = \frac{1}{2}\diag(B_i,B_i)$, so these $2n$ colors
are normalized in $\R^{2n}$.
Apply \Href{Lemma}{lem:large_rainbow_basis} in dimension $2n$.
A lifted basis uses exactly $n$ vectors in each coordinate block,
and its determinant is the product of the two projected
determinants in absolute value. This product is at least
$\delta_{2n}$ from \eqref{eq:delta}. Each factor is at most
$N^{n/2}$ by Hadamard's
inequality \cite[Corollary~7.8.3]{horn2012matrix}, since the
projected vectors have norm at most $\sqrt N$. Thus, both factors
are at least $D_0(n,N)$.
\end{proof}

\section{The colorful quantitative Steinitz theorem}
\label{sec:centered_proof}

We now prove \Href{Theorem}{thm:ellipsoid_steinitz}.
To apply the preceding results, we first replace the original
colors by finite balanced collections and put them in the common
position of \Href{Lemma}{lem:normalization}. This preparation
loses explicit factors depending only on $d$, but produces
normalized vectors with bounded lengths and controlled second
moments.

The same strategy of passing to a more structured configuration
and controlling the return to the original points appears in
the work of Ivanov and Nasz\'odi \cite{ivanov2024steinitz}
and in the polarity argument of \cite{ivanov_QST_polarity_2025}.
Here the change within each color is a radial contraction,
followed by a scalar rescaling; the common linear map preserves
volume.

\begin{lem}[Radial balancing]
\label{lem:radial_balancing}
If $A\Bd \subset \conv X$, there are $N = 2d(d+1)$ indexed points
$x_j \in X$ and numbers $t_j \in [0,1]$ such that
\begin{equation}
 z_j := t_jx_j,
 \qquad \sum_{j \in \brackets{N}} z_j = 0,
 \qquad
 C := \frac{1}{N}\sum_{j \in \brackets{N}} z_j\otimes z_j
 \succeq \frac{A\transpose{A}}{d^2N^2}.
 \label{eq:radial_covariance}
\end{equation}
Repetitions and zero values of $t_j$ are allowed.
\end{lem}

\begin{proof}
Represent each $\pm Ae_k$, $k \in \brackets{d}$, as a convex
combination of at most $d+1$ points of $X$. Collect these points
into an indexed list of length $N$, padding it if necessary:
\[
 \pm Ae_k = \sum_{j \in \brackets{N}}\alpha_{kj}^{\pm}x_j,
 \qquad \alpha_{kj}^{\pm}\ge0,
 \qquad \sum_{j \in \brackets{N}}\alpha_{kj}^{\pm} = 1.
\]
Put
\[
 t_j := \frac{1}{2d}\sum_{k \in \brackets{d}}
 \parenth{\alpha_{kj}^++\alpha_{kj}^-},
 \qquad
 q_j := \frac{1}{2}
 \parenth{\alpha_{kj}^+-\alpha_{kj}^-}_{k = 1}^{d}.
\]
Then $\sum_{j \in \brackets{N}}t_j = 1$,
$\sum_{j \in \brackets{N}}t_jx_j = 0$, and
$\enorm{q_j} \le d t_j$. For every $v \in \R^d$,
\[
 \transpose{A}v = \sum_{j \in \brackets{N}}\iprod{x_j}{v}q_j.
\]
Minkowski's inequality followed by the Cauchy--Schwarz inequality gives
\[
 \enorm{\transpose{A}v}
 \le d\sum_{j \in \brackets{N}}\abs{\iprod{t_jx_j}{v}}
 \le d\sqrt N\parenth{\sum_{j \in \brackets{N}}\iprod{z_j}{v}^2}^{1/2}.
\]
Squaring proves
$A\transpose{A} \preceq d^2N\sum_{j \in \brackets{N}}z_j\otimes z_j
 = d^2N^2C$.
\end{proof}

\begin{lem}
\label{lem:radial_monotonicity}
Let $x_1, \dots, x_m \in \R^d$ and
$y_j = a_jx_j$, where $0 \le a_j \le 1$ for $j \in \brackets{m}$.
If $0 \in \inter\conv\braces{y_j:j \in \brackets{m}}$, then
\[
 \conv\braces{y_j:j \in \brackets{m}}
 \subset \conv\braces{x_j:j \in \brackets{m}}.
\]
\end{lem}

\begin{proof}
Choose strictly positive $b_j$ with $\sum_{j \in \brackets{m}} b_j = 1$ and
$\sum_{j \in \brackets{m}} b_jy_j = 0$. Such coefficients are obtained by averaging
representations of zero giving positive weight to each prescribed
point. At least one $a_j$ is positive, so
$\sum_{j \in \brackets{m}} b_ja_jx_j = 0$ with $\sum_{j \in \brackets{m}} b_ja_j > 0$. Thus,
$0 \in \conv\braces{x_j:j \in \brackets{m}}$, and each
$y_j \in [0,x_j]$ belongs
to this hull.
\end{proof}

The next inclusion turns positive balance into an inscribed ellipsoid.

\begin{lem}
\label{lem:balanced_ellipsoid}
Let $r_1, \dots, r_m \in \R^d$. If $\lambda_i \ge 0$ for
$i \in \brackets{m}$, $\sum_{i \in \brackets{m}}\lambda_i = 1$, and
$\sum_{i \in \brackets{m}}\lambda_i r_i = 0$, then
\[
 \frac{1}{\sqrt{2}}
 \parenth{\sum_{i \in \brackets{m}}\lambda_i^2 r_i\otimes r_i}^{1/2}\Bd
 \subset \conv\braces{r_i:i \in \brackets{m}}.
\]
\end{lem}

\begin{proof}
For $v \in \R^d$, put $b_i := \lambda_i\iprod{v}{r_i}$ and
$s := \sum_{\substack{i \in \brackets{m}\\b_i>0}}b_i$.
Since $\sum_{i \in \brackets{m}}b_i = 0$, the negative terms have total
absolute value $s$. The sum of squares of nonnegative numbers is
at most the square of their sum, so
\[
 \sum_{i \in \brackets{m}}b_i^2
 \le \parenth{\sum_{\substack{i \in \brackets{m}\\b_i > 0}}b_i}^2
      +\parenth{\sum_{\substack{i \in \brackets{m}\\b_i < 0}}(-b_i)}^2
 = 2s^2.
\]
Moreover, $\sum_{\substack{i \in \brackets{m}\\b_i > 0}}\lambda_i \le 1$ gives
$s \le \max_{i \in \brackets{m}}\iprod{v}{r_i}$.
The maximum is nonnegative because zero belongs to the hull.
Comparing support functions gives the inclusion.
\end{proof}

\begin{proof}[Proof of \Href{Theorem}{thm:ellipsoid_steinitz}]
Apply \Href{Lemma}{lem:radial_balancing} to each color with
$N = 2d(d+1)$, and write
$C_i := N^{-1}\sum_{j \in \brackets{N}}z_{ij}\otimes z_{ij}$.
Apply \Href{Lemma}{lem:normalization} to these operators.
With the resulting positive operator $H$, put
$s_i := \sqrt{\iprod{C_i}{H}}$ and
$y_{ij} := H^{1/2}z_{ij}/s_i$. These colors satisfy
\eqref{eq:normalized_colors}. From \eqref{eq:radial_covariance}
and the determinant assumption,
\[
 s_i \ge s_0 := \frac{1}{\sqrt d\,N}.
\]

Put $D_0 := D_0(d,N) = \delta_{2d}N^{-d/2}$, as in
\eqref{eq:basis_scale}, and use \Href{Lemma}{lem:two_groups}
to partition the colors into $I_1$ and $I_2$.
Keep a rainbow basis $B_1$ from $I_1$ with
$\abs{\det B_1} \ge D_0$, and denote its vectors by $r_i$,
$i \in I_1$. Put
\[
 v := \frac{1}{d}\sum_{i \in I_1}r_i,
 \qquad \enorm{v} \le \sqrt N.
\]
Apply \Href{Theorem}{thm:conic_completion} to the colors in $I_2$
and the vector $-v$, with $R = \sqrt N$ and $D = D_0$.
We obtain vectors $r_i$, $i \in I_2$, with
\[
 -v = \sum_{i \in I_2}a_i r_i,
 \qquad 0 \le a_i \le \frac{N^{d/2}}{D_1},
 \qquad D_1 := \frac{D_0}{N^d}.
\]
Set $a_i := 1/d$ for $i \in I_1$. Then
\[
 \sum_{i \in \brackets{2d}}a_i r_i = 0,
 \qquad
 A_* := \sum_{i \in \brackets{2d}}a_i
 \le 1+\frac{dN^{d/2}}{D_1}
 \le \frac{2dN^{d/2}}{D_1}.
\]
For the last inequality, note that $D_1 \le N^{d/2}$.
The normalized weights $\lambda_i := a_i/A_*$ therefore satisfy
\[
 \lambda_i \ge \ell := \frac{D_1}{2d^2N^{d/2}}
 \qquad\text{for every }i \in I_1.
\]
No lower bound for the weights in $I_2$ is needed. Indeed, for
$M_\lambda := \sum_{i \in \brackets{2d}}\lambda_i^2r_i\otimes r_i$,
the first basis gives
\[
 M_\lambda \succeq \ell^2 B_1\transpose{B_1},
 \qquad \sqrt{\det M_\lambda} \ge \ell^d D_0,
\]
where the determinant estimate follows from monotonicity in the
Loewner order \cite[Corollary~7.7.4(c),(e)]{horn2012matrix}.
Equivalently, for decreasingly ordered eigenvalues and singular
values, $\lambda_j(M_\lambda) \ge \ell^2\sigma_j(B_1)^2$ for
$j \in \brackets{d}$, and one multiplies these inequalities.
\Href{Lemma}{lem:balanced_ellipsoid} now puts an origin-centered
ellipsoid of volume at least $\kappa_d 2^{-d/2}\ell^dD_0$
inside the normalized rainbow hull.

Let $x_i$ be the corresponding original points and $t_i$ their
radial factors. Then
\[
 s_0r_i = \frac{s_0t_i}{s_i}H^{1/2}x_i,
 \qquad 0 \le \frac{s_0t_i}{s_i} \le 1.
\]
The normalized hull has zero in its interior.
\Href{Lemma}{lem:radial_monotonicity} and $\det H^{1/2} = 1$ therefore
return an origin-centered ellipsoid in the original rainbow hull,
whose defining operator has absolute determinant at least
\[
 c_d := 
 2^{-d/2}s_0^d\ell^d D_0
 = 
 \frac{\delta_{2d}^{d+1}}
 {2^{3d/2}d^{5d/2}N^{2d^2+3d/2}}.
\]
By \eqref{eq:delta}, $\delta_{2d} \ge e^{-d}$.
Since $N = 2d(d+1)$, we obtain
$c_d = d^{-O(d^2)}$.
\end{proof}

\subsection{Ellipsoids with different centers}

The following question remains open.

\begin{conj}
\label{conj:different_centers}
For every $d\ge1$ there is $\widetilde c_d>0$ with the following
property. Let $X_1, \dots, X_{2d} \subset \R^d$, and suppose that,
for every $i \in \brackets{2d}$, there are $c_i\in\R^d$ and an
invertible operator $A_i$ such that
\[
 c_i+A_i\Bd \subset \conv X_i,
 \qquad \abs{\det A_i} \ge 1.
\]
Then there are a rainbow set $R$, a point
$c \in C := \conv\braces{c_i:i \in \brackets{2d}}$, and an
invertible operator $A$ such that
\[
 c+A\Bd \subset \conv R,
 \qquad \abs{\det A} \ge \widetilde c_d.
\]
\end{conj}

\section{The colorful quantitative Helly theorem}
\label{sec:affine_proof}

We prove the Helly theorem by solving the following finite
halfspace problem. The general-position assumption ensures
that all the interpolation points used in the affine conic
identity are well defined.

\begin{thm}[The reduced problem]
\label{thm:reduced}
For each $i \in \brackets{2d}$, let
\[
 K_i = \bigcap_{j \in \brackets{n_i}}H_{ij},
 \qquad
 H_{ij} = \braces{x \in \R^d \st \iprod{\nu_{ij}}{x} \le b_{ij}},
 \qquad n_i \le 2d,
\]
be a bounded polytope with non-empty interior and volume at most one.
Assume that every $d$ distinct indexed defining normals, taken
from all colors, are linearly independent. Then there are $H_i \in 
\braces{H_{ij}:j \in \brackets{n_i}}$ such that
\[
 \vol{d}\parenth{\bigcap_{i \in \brackets{2d}}H_i} \le A_d,
 \qquad A_d = d^{O(d^2)}.
\]
\end{thm}

Section~\ref{sec:polyhedral_reduction} justifies this reduction. We then describe the duality with the Steinitz
construction, first with a common center and then using
balanced normals together with their heights.

\subsection{Justifying the polyhedral reduction}
\label{sec:polyhedral_reduction}

The reduction uses outer approximation by polytopes and small
perturbations of finitely many facet normals. For the standard
approximation results, see Schneider \cite{schneider2014convex}.
Every auxiliary halfspace retains the label of an original set
which it contains.

{We use the following form of the quantitative volume
Helly theorem, proved by Brazitikos \cite{brazitikos2017brascamp}
and recovered by Almendra-Hern\'andez, Ambrus, and Kendall
\cite[Theorem~4]{almendra2022quantitative}.}

\begin{prp}
\label{prp:one_color_volume}
There is an absolute constant $C > 0$ such that, with
$Q_d := (Cd)^{3d/2}$, every finite family of halfspaces
whose intersection is a bounded convex body of volume at most one
has a subfamily of at most $2d$ members whose intersection has
volume at most $Q_d$.
\end{prp}

The usual approximation and limiting argument gives the following
reduction.

\begin{lem}[Reduction to the polyhedral problem]
\label{lem:reduction}
\Href{Theorem}{thm:reduced} implies
\Href{Theorem}{thm:colorful_volume_helly} with
$\gamma_d = (Q_d A_d)^{-1}$.
\end{lem}

\begin{proof}
We prove the selection form. First suppose that all sets are
non-empty and compact, and that every monochromatic intersection
has volume at most one. By enlarging the sets if necessary and
approximating them from outside by polytopes, we obtain in each
color a finite family of labeled halfspaces whose intersection
has non-empty interior and volume at most $1+\varepsilon$.
\Href{Proposition}{prp:one_color_volume} reduces each family to
at most $2d$ halfspaces, with intersection volume at most
$Q_d(1+\varepsilon)$.

Perturb the finitely many normals into general position. The
monochromatic intersections remain bounded with non-empty
interior, and their volumes converge to the original volumes.
After a common rescaling, apply \Href{Theorem}{thm:reduced}.
Let the perturbations tend to zero along a subsequence with the
same selected labels. Fatou's lemma bounds the resulting rainbow
intersection by $A_dQ_d(1+\varepsilon)$. Its parent sets have no
larger intersection. Let $\varepsilon$ decrease to zero.
For arbitrary convex sets, first pass to closures and truncate
by expanding balls, then use continuity of volume from below
and the finiteness of the original rainbow selections. An empty
member already gives the desired selection.
\end{proof}

\subsection{Notation and the dual construction}
\label{sec:dual_construction}
\label{sec:halfspace_duality}

The use of translations in the dual construction follows the
polarity approach of \cite{ivanov_QST_polarity_2025}, where the
center is chosen so that the polar vertices are balanced.
Here, we keep track of the change of center through the heights
of the defining halfspaces.

We now introduce the notation used in the proof of the reduced
theorem, \Href{Theorem}{thm:reduced}. First we describe halfspaces
and polarity, then define the weighted facet data and the common
normalization. The resulting normal--height pairs will be used
throughout the proof.

For $y \in \R^d\setminus\braces{0}$ and $h \in \R$, write
\[
 \mathcal H(y,h) :=
 \braces{x \in \R^d:\iprod{y}{x} \le h}.
\]
We call $h$ the \emph{height} of the inequality with normal $y$.
Its signed Euclidean offset is $h/\enorm{y}$. Multiplying
$y$ and $h$ by the same positive scalar leaves the halfspace
unchanged. Translation and an invertible linear map $S$ give
\begin{equation}
 \mathcal H(y,h)+c = \mathcal H(y,h+\iprod{y}{c}),
 \qquad S\mathcal H(y,h) = \mathcal H(S^{-T}y,h).
 \label{eq:halfspace_changes}
\end{equation}

If $K = \bigcap_{j \in \brackets{m}}\mathcal H(y_j,h_j)$ is a bounded
polytope and $p \in \inter K$, put
\begin{equation}
 h_j(p) := h_j-\iprod{y_j}{p} > 0,
 \qquad v_j(p) := \frac{y_j}{h_j(p)},
 \qquad (K-p)^\circ = \conv\braces{v_j(p):j \in \brackets{m}}.
 \label{eq:polar_vertices}
\end{equation}
Here $L^\circ := \braces{v:\iprod{v}{x} \le 1\text{ for all }x \in L}$
denotes polarity about zero. Formula~\eqref{eq:polar_vertices}
follows by writing the inequalities of $K-p$ with right-hand
side one. The choice of $p$ enters only through the denominators.

\begin{dfn}
\label{dfn:facet_data}
Let $K$ be a bounded polytope with non-empty interior. Write its
irredundant facet description as
\[
 K = \bigcap_{j \in \brackets{m}}\mathcal H(\nu_j,b_j),
 \qquad \enorm{\nu_j} = 1.
\]
Let $F_j$ be the facet with outer normal $\nu_j$, and put
$f_j := \vol{d-1}F_j$, $V := \vol{d}K$, and
$\mathcal A(K) := \sum_{j \in \brackets{m}}f_j$.
Define its weighted facet data by
\begin{equation}
 \mu_j := \frac{f_j}{dV},
 \qquad z_j := \mu_j\nu_j,
 \qquad q_j := \mu_j b_j,
 \qquad C := \frac{1}{m}\sum_{j \in \brackets{m}}z_j\otimes z_j.
 \label{eq:facet_data}
\end{equation}
Thus, $K = \bigcap_{j \in \brackets{m}}\mathcal H(z_j,q_j)$.
The operator $C$ is positive definite because the facet normals
span $\R^d$. For $p \in \inter K$, we also use the weights
\begin{equation}
 \theta_j(p) := q_j-\iprod{z_j}{p}
 = \mu_j\parenth{b_j-\iprod{\nu_j}{p}} > 0.
 \label{eq:facet_radial_weights}
\end{equation}
In these coordinates, $v_j(p) = z_j/\theta_j(p)$.
\end{dfn}

\begin{dfn}
\label{dfn:reduced_construction}
Let $K_1, \dots, K_{2d}$ be bounded polytopes with non-empty
interior, at most $N$ facets each, and volumes $V_i \le 1$.
Attach the color index $i$ to the notation of
\Href{Definition}{dfn:facet_data}, writing $n_i$ for the number
of facets of $K_i$. Thus, for $i \in \brackets{2d}$ and
$j \in \brackets{n_i}$, we have
\[
 \mu_{ij} = \frac{f_{ij}}{dV_i},
 \qquad (z_{ij},q_{ij}) = \mu_{ij}(\nu_{ij},b_{ij}),
 \qquad C_i = \frac{1}{n_i}\sum_{j \in \brackets{n_i}}z_{ij}\otimes z_{ij}.
\]
Choose the positive definite operator $H$ with $\det H = 1$
given by \Href{Lemma}{lem:normalization} for
$C_1, \dots, C_{2d}$, and put
\begin{equation}
 s_i := \sqrt{\iprod{C_i}{H}},
 \qquad y_{ij} := \frac{H^{1/2}z_{ij}}{s_i},
 \qquad h_{ij} := \frac{q_{ij}}{s_i}.
 \label{eq:normalized_dual_pairs}
\end{equation}
We write
\[
 Y_i := \braces{y_{ij}:j \in \brackets{n_i}},
 \qquad \beta_i := \frac{1}{n_i}\sum_{j \in \brackets{n_i}}h_{ij}.
\]
All these sets are indexed. The normal $y_{ij}$ retains the
height $h_{ij}$ and the label of its original halfspace.

The common coordinate change is $x' := H^{-1/2}x$.
It preserves volume and gives
\[
 K_i' := H^{-1/2}K_i
 = \bigcap_{j \in \brackets{n_i}}\mathcal H(y_{ij},h_{ij}).
\]
\end{dfn}

For the polytopes in \Href{Theorem}{thm:reduced}, discard redundant
inequalities and rescale the remaining normals to unit length.
Then \Href{Definition}{dfn:reduced_construction} applies with
$N = 2d$. Every retained inequality keeps its original color and
label. The lemmas below establish the quantitative properties of
this construction.

The formula for $K_i'$ follows because
$\iprod{z_{ij}}{x} \le q_{ij}$ becomes
$\iprod{H^{1/2}z_{ij}}{x'} \le q_{ij}$; we then divide by $s_i$.
The construction is therefore
\[
 (\nu_{ij},b_{ij})
 \longmapsto (z_{ij},q_{ij}) = \mu_{ij}(\nu_{ij},b_{ij})
 \longmapsto (y_{ij},h_{ij})
 = \frac{1}{s_i}(H^{1/2}z_{ij},q_{ij}).
\]
For $p_i \in \inter K_i$ and $p_i' := H^{-1/2}p_i$, its effect
on the polar vertices is
\[
 \frac{y_{ij}}{h_{ij}-\iprod{y_{ij}}{p_i'}}
 = H^{1/2}\frac{z_{ij}}{q_{ij}-\iprod{z_{ij}}{p_i}}
 = H^{1/2}v_{ij}(p_i).
\]
The factor $s_i$ cancels in the polar vertex. We keep it in the
normal--height pair to normalize the second moments.

\subsection{Minkowski identities and volume estimates}
\label{sec:dual_normalization}

Minkowski's equilibrium identity and the signed pyramid-volume
formula give, in the notation of \Href{Definition}{dfn:facet_data},
\begin{equation}
 \sum_{j \in \brackets{m}}f_j\nu_j = 0,
 \qquad \sum_{j \in \brackets{m}}f_jb_j = dV.
 \label{eq:facet_identities}
\end{equation}
For the surface-area measure and its equilibrium condition, see
\cite[Sections~10.1 and~18.2]{GruberBook}.
The second identity holds
for every position of the origin, even when some $b_j$ are negative.
The classical isoperimetric inequality gives
\begin{equation}
 \mathcal A(K) \ge d\kappa_d^{1/d}V^{(d-1)/d};
 \label{eq:facet_isoperimetry}
\end{equation}
see \cite[Section~8.3, Theorem~8.7]{GruberBook}.

Under an invertible linear map $S$, the facet data satisfy
\[
 \nu_j' = \frac{S^{-T}\nu_j}
 {\enorm{S^{-T}\nu_j}},
 \qquad
 f_j' = \abs{\det S}\enorm{S^{-T}\nu_j}f_j.
\]
For the second formula, compare a pyramid with base $F_j$ and
its image: the volume changes by $\abs{\det S}$ and the height
by $1/\enorm{S^{-T}\nu_j}$. Consequently,
\begin{equation}
 f_j'\nu_j' = \abs{\det S}S^{-T}(f_j\nu_j).
 \label{eq:facet_transform}
\end{equation}

\begin{lem}
\label{lem:facet_weights}
For the weighted facet data of \Href{Definition}{dfn:facet_data},
\begin{equation}
 \sum_{j \in \brackets{m}}z_j = 0,
 \qquad \sum_{j \in \brackets{m}}q_j = 1,
 \qquad
 \sqrt{\det C} \ge \frac{\kappa_d}{m^d d^{d/2}V}.
 \label{eq:facet_balance}
\end{equation}
Translating $K$ by $c$ leaves $z_j$ unchanged and replaces $q_j$
by $q_j+\iprod{z_j}{c}$.
\end{lem}

\begin{proof}
Divide \eqref{eq:facet_identities} by $dV$ to obtain the two
balance identities. The translation rule follows from
\eqref{eq:halfspace_changes}.

Put $M := \sum_{j \in \brackets{m}}z_j\otimes z_j$ and map $K$ by
$S := M^{1/2}$. By \eqref{eq:facet_transform}, the normalized
facet vectors of $K' := SK$ are
$z_j' = S^{-T}z_j$. Hence
$\sum_{j \in \brackets{m}}z_j'\otimes z_j' = \id_d$ and
$\sum_{j \in \brackets{m}}\enorm{z_j'}^2 = d$.
Writing $V' := \vol{d}K' = V\sqrt{\det M}$, the Cauchy--Schwarz inequality gives
\[
 \mathcal A(K') = dV'\sum_{j \in \brackets{m}}\enorm{z_j'}
 \le dV'\sqrt m\parenth{\sum_{j \in \brackets{m}}\enorm{z_j'}^2}^{1/2}
 = dV'\sqrt{md}.
\]
Together with \eqref{eq:facet_isoperimetry}, this yields
$V' \ge \kappa_d/(md)^{d/2}$. Since
$\sqrt{\det C} = m^{-d/2}\sqrt{\det M}$, the bound follows.
\end{proof}

The weights in \eqref{eq:facet_radial_weights} satisfy
$\sum_{j \in \brackets{m}}\theta_j(p) = 1$. In fact, $\theta_j(p)$ is
the volume of the pyramid with apex $p$ and base $F_j$, divided
by $V$. Thus,
\[
 z_j = \theta_j(p)v_j(p),
 \qquad 0 < \theta_j(p) \le 1,
 \qquad \sum_{j \in \brackets{m}}\theta_j(p)v_j(p) = 0.
\]
These are balanced radial contractions of the polar vertices,
just as in \Href{Lemma}{lem:radial_balancing}. Changing the pole
changes the vertices and their contraction factors, but leaves
the vectors $z_j$ unchanged. The preceding determinant estimate
is the dual counterpart of the second-moment estimate in that lemma.

\begin{lem}[Normalized halfspaces]
\label{lem:normalized_halfspaces}
For the polytopes and normal--height pairs of
\Href{Definition}{dfn:reduced_construction}, the colors
$Y_1, \dots, Y_{2d}$ are normalized in the sense of
\Href{Definition}{dfn:normalized_colors}. For every
$i \in \brackets{2d}$, their average heights satisfy
\begin{equation}
 0 < \beta_i \le \kappa_d^{-1/d}.
 \label{eq:height_average}
\end{equation}
The construction preserves general position of the normals.
\end{lem}

\begin{proof}
Normalization follows from the construction after
\Href{Lemma}{lem:normalization}. By that lemma and
\eqref{eq:facet_balance},
\[
 s_i \ge \sqrt d\parenth{\det C_i}^{1/(2d)}
 \ge \frac{\kappa_d^{1/d}}{n_i V_i^{1/d}},
 \qquad \beta_i = \frac{1}{n_i s_i} \le \kappa_d^{-1/d}.
\]
All rescalings are positive and the common normal map $H^{1/2}$
is invertible, so linear independence is preserved.
\end{proof}

\subsection{The idea of the proof: centers and heights}
\label{sec:helly_idea}

Suppose first that the John ellipsoids of the monochromatic
intersections share a center $p$. John's theorem
\cite{John} gives
$E_i \subset K_i \subset p+d(E_i-p)$; see also
\cite[Section~11.1]{GruberBook}. Translate $p$ to zero and
use the polar vertices from \eqref{eq:polar_vertices}. Their
colors $X_i$ satisfy
\[
 \conv X_i = K_i^\circ \supset \frac{1}{d}E_i^\circ.
\]
Since $\vol{d}E_i \le 1$ and
$(\vol{d}E_i)(\vol{d}E_i^\circ) = \kappa_d^2$, each dual
color contains an origin-centered ellipsoid of volume at least
$\kappa_d^2/d^d$. By \Href{Theorem}{thm:ellipsoid_steinitz}
and a common rescaling, a rainbow hull contains an ellipsoid
of volume at least $c_d\kappa_d^2/d^d$. Taking polars bounds
the corresponding rainbow intersection by $d^d/c_d$.

For different centers, we use the normal--height pairs of
\Href{Definition}{dfn:reduced_construction}. If some rainbow
intersection is empty, the desired selection is immediate.
We therefore assume that all rainbow intersections are non-empty.
The facet weights balance the normals, and the common positive
operator normalizes their second moments. The volume assumption
becomes the average-height bound \eqref{eq:height_average}.
Translating color $i$ by $c_i$ changes its heights to
$h_y+\iprod{y}{c_i}$, but leaves their average unchanged because
$\sum_{y \in Y_i}y = 0$. Individual heights may have either sign
and need not be bounded.

We partition the colors into two groups. For a unit direction
$u$, we choose a covering rainbow basis from the first group
and a basis covering $-u$ from the second. Their conic coefficients give
\[
 u = \sum_{i \in I_1}a_i r_i,
 \qquad -u = \sum_{i \in I_2}a_i r_i.
\]
The selected halfspaces contain their intersection in a strip
orthogonal to $u$. Its width is
\[
 L := \sum_{i \in I_1}a_i h_{r_i}
       +\sum_{i \in I_2}a_i h_{r_i}.
\]
Indeed, the first sum is the upper supporting level in direction
$u$ for the first group, and the negative of the second sum is
the lower supporting level for the second group. Small $L$ has
this concrete geometric meaning: the rainbow intersection is
contained in a narrow strip. These levels need not be attained
in the full rainbow intersection, so its width may be smaller.

Alas, the strip alone does not control volume. We also need the selected
weighted normals to span all directions quantitatively. This is
measured by the determinant of
\[
 M := \sum_{i \in \brackets{2d}}a_i^2r_i\otimes r_i.
\]
The dual ellipsoid lemma converts a bound on $L$ and a lower
bound on $\det M$ into a volume bound. Thus, the coefficients
$a_i$, chosen together with the rainbow, must control both the
heights and the geometry of the normals.

To choose the direction, we take a determinant-weighted average
of these strip widths over all pairs of covering bases. Denote
it by $\overline L(u)$; its finite-sum definition is given below in
\eqref{eq:mean_strip_width}. Here the direction $u$ is fixed,
and the two bases vary independently within their respective
groups. The determinant weights have a total mass independent
of $u$, by the conic covering identity.

We compare the given heights with a second assignment: every
normal of color $i$ receives the constant height $\beta_i$.
For this assignment the average strip width
$\overline L_\beta(u)$ is uniformly bounded. The difference
between the two height assignments sums to zero in each color,
so the affine conic identity gives
\[
 \overline L(u) = \overline L_\beta(u)+\iprod{v}{u}
\]
for one fixed vector $v$, independent of the direction. Thus,
the part of the heights not controlled by their color averages
contributes only a linear function of $u$. This includes the
effect of independent translations of the colors.

The spherical-zone estimate leaves a line separated from every
facet span. Choose its orientation so that $\iprod{v}{u} \le 0$.
The average strip width is then small, and separation keeps all
conic coefficients away from zero. No bound on the size of $v$
is needed: reversing the direction changes the sign of its
contribution without affecting separation.

We still have to choose one pair of bases. Nonnegativity of all
strip widths implies that pairs with large width carry at most
half of the total weight. Pairs whose first basis has a small
determinant carry at most a quarter of the weight. A pair remains
after both exclusions. Its width is small, its first basis has
a large determinant, and all its conic coefficients have the
lower bound supplied by separation. The Cauchy--Binet formula
then gives the required lower bound for $\det M$, and the dual
ellipsoid lemma bounds the volume of the rainbow intersection.

\subsection{The dual of the balanced ellipsoid lemma}
\label{sec:dual_ellipsoid}

For a finite selection of normal--height pairs $(r_i,h_i)$,
$i \in \brackets{m}$, and weights $a_i \ge 0$ satisfying
$\sum_{i \in \brackets{m}}a_i r_i = 0$, put
\begin{equation}
 P := \bigcap_{i \in \brackets{m}}\mathcal H(r_i,h_i),
 \qquad M := \sum_{i \in \brackets{m}}a_i^2r_i\otimes r_i,
 \qquad L := \sum_{i \in \brackets{m}}a_i h_i.
 \label{eq:selected_dual_data}
\end{equation}
The balance relation makes $L$ independent of the choice of
origin. More concretely, for any $p \in P$,
\[
 L = \sum_{i \in \brackets{m}}a_i
 \parenth{h_i-\iprod{r_i}{p}}.
\]
Thus, $L$ is also the weighted sum of the nonnegative slacks of
the selected inequalities at $p$. The operator $M$ measures
their normals in all directions, through
\[
 \iprod{Mx}{x} = \sum_{i \in \brackets{m}}a_i^2\iprod{r_i}{x}^2.
\]

When $M \succ 0$ and $L > 0$, the volume functional is
\begin{equation}
 \Phi := \frac{\sqrt{\det M}}{L^d}.
 \label{eq:helly_functional}
\end{equation}
It is unchanged by a common scaling of the weights or a common
translation of the halfspaces. The next lemma is the polar
form of \Href{Lemma}{lem:balanced_ellipsoid}: the inner quadratic
ellipsoid becomes an outer ellipsoid for the intersection.

\begin{lem}[The dual balanced ellipsoid]
\label{lem:dual_ellipsoid}
For the data in \eqref{eq:selected_dual_data}, assume that $P$
is non-empty and $M$ is positive definite. Then $L \ge 0$ and
\begin{equation}
 \vol{d}P \le
 \frac{2^{d/2}\kappa_d L^d}{\sqrt{\det M}}.
 \label{eq:dual_ellipsoid_volume}
\end{equation}
If $P$ has non-empty interior, then $L > 0$ and, for every
$p \in \inter P$,
\[
 P \subset p+\sqrt{2}L M^{-1/2}\Bd.
\]
\end{lem}

\begin{proof}
For $x \in P$, summing $a_i\iprod{r_i}{x} \le a_i h_i$
over $i \in \brackets{m}$ gives $0 \le L$.
If $P$ has empty interior, its volume is zero and the volume
bound follows. Otherwise fix $p \in \inter P$ and set
\[
 \tau_i := h_i-\iprod{r_i}{p} > 0,
 \qquad v_i := \frac{r_i}{\tau_i},
 \qquad \lambda_i := \frac{a_i\tau_i}{L}.
\]
Here $L = \sum_{i \in \brackets{m}}a_i\tau_i > 0$. The polar points
$v_i$ satisfy
\[
 \sum_{i \in \brackets{m}}\lambda_i = 1,
 \qquad \sum_{i \in \brackets{m}}\lambda_i v_i = 0,
 \qquad \sum_{i \in \brackets{m}}\lambda_i^2v_i\otimes v_i = \frac{M}{L^2}.
\]
The slacks $\tau_i$ cancel in the last identity. Although the
polar points depend on $p$, their weighted quadratic operator
depends only on the selected normals, heights, and coefficients.
The  positive balance also puts zero in the convex hull of the
$v_i$. Thus, the polar of $P-p$ is that convex hull, even before
boundedness of $P$ has been established. By
\Href{Lemma}{lem:balanced_ellipsoid},
\[
 \frac{1}{\sqrt{2}L}M^{1/2}\Bd
 \subset \conv\braces{v_i:i \in \brackets{m}} = (P-p)^\circ.
\]
Taking polars gives the asserted containment and its volume.
\end{proof}

Thus, as in the centered Steinitz proof, we seek a rainbow
selection with a large determinant; here we must also control
the height sum $L$. The affine conic identity below provides the
additional control of $L$.

\subsection{Affine conic sums and supporting vertices}
\label{sec:conic_identities}

We again state the colorful identity in ambient dimension $n$,
using $\mathcal H(y,h)$ for the same halfspace formula in $\R^n$.
For $n$ balanced colors, use $\mathcal B(u)$ from
\Href{Definition}{dfn:rainbow_bases}. Attach a real height $h_y$ to
each indexed normal $y$. For a rainbow basis
$B = [r_1\ \dots\ r_n]$, write
\[
 h_B := \transpose{\parenth{h_{r_1}, \dots, h_{r_n}}},
 \qquad p_B := B^{-T}h_B.
\]
The point $p_B$ is the common point of the $n$ supporting
hyperplanes $\iprod{r_i}{x} = h_{r_i}$.
If $u \in \poscone{B}$ and $a := B^{-1}u \ge 0$, then
\[
 \sup_{x \in \bigcap_{i \in \brackets{n}}\mathcal H(r_i,h_{r_i})}
 \iprod{u}{x}
 = \iprod{u}{p_B}
 = \sum_{i \in \brackets{n}}a_i h_{r_i}.
\]
Indeed, adding the inequalities with coefficients $a_i$ gives
the upper bound, and $p_B$ attains it.

The determinant-weighted sum of these support values is
\begin{equation}
 G_h(u) := \sum_{B \in \mathcal B(u)}
 \abs{\det B}\,\iprod{u}{p_B}
 = \sum_{B \in \mathcal B(u)}
 \abs{\det B}\,\iprod{h_B}{B^{-1}u}.
 \label{eq:affine_conic_sum}
\end{equation}
In \Href{Lemma}{lem:cone_density}, balance makes the sum of
determinant weights constant. The next lemma gives the affine
counterpart: when the heights also sum to zero in each color,
the gradient of $G_h$ is constant. We will apply it to the height
differences $h_y - \beta_i$ in color $i$. This allows us to replace
the original heights by their color averages, at the cost of
one linear function of the direction.

\begin{lem}
\label{lem:affine_conic_identity}
Under the balance assumptions of \Href{Lemma}{lem:cone_density},
suppose additionally that every rainbow selection of $n$ vectors
is linearly independent. If $\sum_{y \in Y_i}h_y = 0$ in each
color, then $G_h$ agrees with a linear function on all generic
directions. Consequently, two height assignments with the same
sum in every color have conic sums differing by a linear function.
\end{lem}

\begin{proof}
For $n=1$, the slopes of $G_h$ on the positive and negative
rays are respectively $\sum_{y \in Y_1,\,y>0} h_y$ and
$-\sum_{y \in Y_1,\,y<0} h_y$.
They agree because $\sum_{y\in Y_1}h_y=0$; zero normals are excluded
by the independence assumption. Thus, assume $n \ge 2$.

In any neighborhood where $\mathcal B(u)$ is fixed, the function
$G_h$ is linear, with gradient
\[
 g = \sum_{B \in \mathcal B(u)}\abs{\det B}\,p_B.
\]
We compare these gradients under the normal displacements used
in the proof of \Href{Lemma}{lem:cone_density}.

Use its notation $H,z,\nu,F = (f_i)_{i \in I},k$ for one
of the indexed facets, and let $\omega_F$ be the volume of
the parallelotope spanned by its vectors. There is a
unique vector $p_F \in H$ with
$\iprod{p_F}{f_i} = h_{f_i}$ for every $i \in I$. For a
completion $y \in Y_k$, set $s_y := \iprod{\nu}{y}$. Every
$s_y$ is nonzero by the independence assumption, and solving
the interpolation equations gives
\[
 p_{B_F(y)}
 = p_F+\frac{h_y-\iprod{p_F}{y}}{s_y}\nu.
\]
Put $Y_k^\pm := \braces{y \in Y_k:\pm s_y > 0}$.
By \eqref{eq:conic_transition}, in passing from $z-t\nu$ to
$z+t\nu$, all completions from $Y_k^+$ enter the sum and all
completions from $Y_k^-$ leave it. Thus, the change contributed
by the whole facet is
\[
 D_F := 
 \sum_{y \in Y_k^+}\abs{\det B_F(y)}\,p_{B_F(y)}
 -\sum_{y \in Y_k^-}\abs{\det B_F(y)}\,p_{B_F(y)}.
\]
Using $\abs{\det B_F(y)} = \omega_F\abs{s_y}$ yields
\[
 D_F = \omega_F\sum_{y \in Y_k}
 \brackets{s_y p_F+\parenth{h_y-\iprod{p_F}{y}}\nu} = 0.
\]
Indeed, $\sum_{y \in Y_k}s_y = \iprod{\nu}{\sum_{y \in Y_k}y} = 0$,
and
$\sum_{y \in Y_k}\parenth{h_y-\iprod{p_F}{y}} = 0$
by the two balance assumptions.

As proved in \Href{Lemma}{lem:cone_density}, every changing
basis belongs to exactly one of these facet groups.
Summing $D_F = 0$ over them shows that $g$ is unchanged by the
normal displacement. As in that lemma, finiteness now gives
the same $g$ on the whole generic set. Since each local expression
is homogeneous,
$G_h(u) = \iprod{g}{u}$ on that set. Finally, apply this
conclusion to the difference of two height assignments with
equal sums in every color.
\end{proof}

In particular, replacing $h_y$ in color $i$ by
$h_y+\iprod{c_i}{y}$ changes $G_h$ by a linear function.
This describes an independent translation of each color of
halfspaces, since the added heights sum to zero by balance.
The same conclusion holds for any changes of the heights whose
sum is zero in each color, including the replacement of each
height by its color average. Linear independence ensures that all interpolation
points in the gradient calculation are well defined; the
polyhedral reduction supplies this assumption.

\subsection{Choosing a direction and a rainbow}
We return to $2d$ balanced colors $Y_1, \dots, Y_{2d} \subset \R^d$
with heights attached to their normals. For a partition into
groups $I_1,I_2$ of $d$ colors, write $\mathcal B_k(u)$
for the rainbow bases of group $I_k$ whose positive cones
contain $u$. Suppose each group has a rainbow basis. By
\Href{Lemma}{lem:cone_density}, the sums
\[
 S_k := \sum_{B \in \mathcal B_k(u)}\abs{\det B},
 \qquad k\in\brackets{2},
\]
are positive and independent of the generic direction $u$.
The families $\mathcal B_k(u)$ may change with $u$, but their
total determinant weights $S_k$ do not.

Fix a generic unit vector $u$. We consider the finite family
of pairs
\[
 \mathcal P(u) := \mathcal B_1(u)\times\mathcal B_2(-u).
\]
Every pair gives one normal $r_i$ of each color and unique
positive coefficients satisfying
\[
 u = \sum_{i\in I_1}a_i r_i,
 \qquad -u = \sum_{i\in I_2}a_i r_i.
\]
In particular, the $2d$ weighted normals sum to zero. Write
$P(B_1,B_2)$ for the intersection of the selected halfspaces.
Adding their inequalities with these coefficients gives
\[
 \iprod{u}{x}\le\sum_{i\in I_1}a_i h_{r_i}
   =\iprod{u}{p_{B_1}},
 \qquad
 \iprod{-u}{x}\le\sum_{i\in I_2}a_i h_{r_i}
   =\iprod{-u}{p_{B_2}}
\]
for every $x\in P(B_1,B_2)$. Thus,
\[
 P(B_1,B_2) \subset
 \braces{x:\iprod{u}{p_{B_2}} \le \iprod{u}{x}
 \le \iprod{u}{p_{B_1}}}.
\]
Since $u$ is a unit vector, the difference between the two
defining levels is
\[
 L_u(B_1,B_2) := \iprod{u}{p_{B_1}-p_{B_2}}
 = \sum_{i \in \brackets{2d}}a_i h_{r_i}.
\]
If $P(B_1,B_2)$ is non-empty, then $L_u(B_1,B_2) \ge 0$
and this difference is the width of the containing strip.
The two supporting vertices belong to the intersections
defined by their own bases, but need not belong to
$P(B_1,B_2)$. Accordingly, $L_u(B_1,B_2)$ bounds the width
of this intersection in direction $u$ from above.

We now average over the pairs in $\mathcal P(u)$, keeping $u$
fixed. Assign $(B_1,B_2)$ the positive weight
\[
 w_u(B_1,B_2) :=
 \frac{\abs{\det B_1}\abs{\det B_2}}{S_1S_2}.
\]
The two bases range independently over their covering families.
Hence the sum of these product weights factors as
\[
 \sum_{B_1\in\mathcal B_1(u)}
 \sum_{B_2\in\mathcal B_2(-u)}w_u(B_1,B_2)
 =\parenth{\frac{1}{S_1}\sum_{B_1\in\mathcal B_1(u)}\abs{\det B_1}}
  \parenth{\frac{1}{S_2}\sum_{B_2\in\mathcal B_2(-u)}\abs{\det B_2}}
 =1.
\]
Here the second factor uses the same constant $S_2$ at the
direction $-u$. More specifically, fixing either basis and
summing over the other gives
\[
 \sum_{B_2\in\mathcal B_2(-u)}w_u(B_1,B_2)
   =\frac{\abs{\det B_1}}{S_1},
 \qquad
 \sum_{B_1\in\mathcal B_1(u)}w_u(B_1,B_2)
   =\frac{\abs{\det B_2}}{S_2}.
\]
These identities explain how to compute the average width.
The upper level $\iprod{u}{p_{B_1}}$ depends only on $B_1$;
the negative of the lower level is
$\iprod{-u}{p_{B_2}}$ and depends only on $B_2$. Summing the two
contributions separately yields
\begin{equation}
\begin{aligned}
 \overline L(u)
 &:= \sum_{B_1 \in \mathcal B_1(u)}
     \sum_{B_2 \in \mathcal B_2(-u)}
     w_u(B_1,B_2)L_u(B_1,B_2)\\
 &= \frac{1}{S_1}\sum_{B_1 \in \mathcal B_1(u)}
       \abs{\det B_1}\iprod{u}{p_{B_1}}
    +\frac{1}{S_2}\sum_{B_2 \in \mathcal B_2(-u)}
       \abs{\det B_2}\iprod{-u}{p_{B_2}}
  = \frac{G_{h,1}(u)}{S_1}+\frac{G_{h,2}(-u)}{S_2},
\end{aligned}
 \label{eq:mean_strip_width}
\end{equation}
where $G_{h,k}$ denotes \eqref{eq:affine_conic_sum} for group
$I_k$. In particular, the second conic sum is evaluated at $-u$,
which accounts for the plus sign in the last line.

This choice of weights serves two purposes. It expresses the
average through the conic sums, so the affine conic identity
can control the heights. It also makes the total weight of
pairs containing a basis with a small determinant small. We
will use both facts to choose a pair whose strip width and
determinant are simultaneously controlled.

We first choose a line away from every facet span. Both
orientations will remain available to control the average.
Separation has a second role: in every covering basis, it
prevents any coefficient of the direction from becoming too
small. This will keep the weighted normals quantitatively
spanning after the pair is selected.
\begin{lem}
\label{lem:separated_direction}
Given at most $2dN$ vectors of norm at most $\sqrt N$ in
$\R^d$, there is a unit vector $u_0$ at distance at least
\begin{equation}
 \eta := \eta_{d,N} := 
 \frac{1}{12\sqrt d\,(2dN)^{d-1}}
 \label{eq:separation_scale}
\end{equation}
from the span of every $d-1$ vectors in the configuration.
For $u \in \braces{u_0,-u_0}$ and any basis $B$ with
$u = \sum_{j \in \brackets{d}} a_j r_j$, $a_j > 0$, one has
\begin{equation}
 \frac{\eta}{\sqrt N} \le a_j
 \le \frac{N^{(d-1)/2}}{\abs{\det B}}.
 \label{eq:coefficient_bounds}
\end{equation}
\end{lem}

\begin{proof}
Let $\sigma$ be the standard Haar probability measure on $\Sd$.
The spherical-cap estimates of B\"or\"oczky and Wintsche
\cite{boroczky2003covering} give, for a unit vector $e$ and
$0 \le \omega \le 1$,
\[
 \sigma\braces{u \in \Sd:\abs{\iprod{u}{e}} \le \omega}
 \le \sqrt{2d}\,\omega.
\]
For $d \le 2$, this estimate also follows by a direct computation.
A proper linear subspace is contained in a hyperplane, so the
same bound applies to its $\omega$-neighborhood on the sphere.
There are at most $(2dN)^{d-1}$ spans to avoid. The measure of
the union of their $\eta$-neighborhoods is at most
\[
 (2dN)^{d-1}\sqrt{2d}\,\eta = \frac{\sqrt 2}{12} < 1.
\]
Choose $u_0$ outside this union. The union is centrally symmetric,
so $-u_0$ satisfies the same separation bound.

Project $u = \sum_{k \in \brackets{d}}a_k r_k$ orthogonally to the
complement of $\Lin\braces{r_k:k \ne j}$. Its projected length
is at least $\eta$, so $a_j\enorm{r_j} \ge \eta$.
This gives the lower bound in \eqref{eq:coefficient_bounds}.
Cramer's rule and Hadamard's inequality give the upper bound,
since $\enorm{u} = 1$ and every $\enorm{r_k} \le \sqrt N$.
\end{proof}

\begin{lem}
\label{lem:joint_selection}
Let $Y_1, \dots, Y_{2d} \subset \R^d$ be normalized, with at most
$N$ vectors per color. Suppose every rainbow set of $d$
vectors is independent. Assign heights $h_y$ satisfying
\[
 0 \le \beta_i := \frac{1}{\card{Y_i}}\sum_{y \in Y_i}h_y \le b,
 \qquad b>0.
\]
Assume every rainbow intersection of the halfspaces
$\braces{x:\iprod{y}{x} \le h_y}$ is non-empty.
Put $D_0 := D_0(d,N) = \delta_{2d}N^{-d/2}$, as in
\eqref{eq:basis_scale}. Let $\eta = \eta_{d,N}$ be as in
\eqref{eq:separation_scale}, and set
\begin{equation}
 K_0 := \frac{2dbN^dN^{(d-1)/2}}{D_0}.
 \label{eq:height_scale}
\end{equation}
Then there are $r_i \in Y_i$ and $a_i>0$ such that
\[
 \sum_{i \in \brackets{2d}}a_i r_i = 0,
 \qquad 0 \le L := \sum_{i \in \brackets{2d}}a_i h_{r_i}\le2K_0,
\]
and
\begin{equation}
 M := \sum_{i \in \brackets{2d}}a_i^2r_i\otimes r_i,
 \qquad
 \sqrt{\det M} \ge 
 \parenth{\frac{\eta}{\sqrt N}}^d\frac{D_0}{4N^d}.
 \label{eq:joint_determinant}
\end{equation}
\end{lem}

\begin{proof}
Partition the colors into $I_1,I_2$ by
\Href{Lemma}{lem:two_groups}, and use the preceding notation
$\mathcal B_k(u)$, $S_k$, $w_u$, and $L_u$ for these groups.
Each group has a rainbow basis of determinant at least $D_0$.
Evaluating its conic determinant sum at a generic point in
the interior of that basis cone, and then using
\Href{Lemma}{lem:cone_density}, gives
$S_k\ge D_0$ for every generic direction and each
$k\in\brackets{2}$. Every pair of covering bases gives a
non-empty rainbow intersection by assumption, so every strip
width appearing in \eqref{eq:mean_strip_width} is nonnegative.

We first bound the part of the average that depends on the
color averages $\beta_i$. Let $G_{\beta,k}$ be the conic sum
obtained by assigning the constant height $\beta_i$ to every
normal of color $i$. For a covering basis $B$ of group $I_k$,
write $a_i(B,u)$ for the coefficient of its color-$i$ normal in
the representation of $u$. Thus, for a generic unit vector $u$,
\[
 G_{\beta,k}(u)
 =\sum_{B\in\mathcal B_k(u)}\abs{\det B}
       \sum_{i\in I_k}\beta_i a_i(B,u).
\]
All summands are nonnegative. Cramer's rule and Hadamard's
inequality give
$\abs{\det B}a_i(B,u)\le N^{(d-1)/2}$: the relevant determinant
has one column $u$ of length one and $d-1$ columns of length
at most $\sqrt N$. Since there are at most $N^d$ covering bases
and $d$ coefficients in each, we obtain
\[
 0\le\frac{G_{\beta,k}(u)}{S_k}
 \le\frac{1}{S_k}\sum_{B\in\mathcal B_k(u)}
                     \sum_{i\in I_k}bN^{(d-1)/2}
 \le\frac{dbN^dN^{(d-1)/2}}{D_0}.
\]

We now compare this with the original heights. In each color,
\[
 \sum_{y\in Y_i}(h_y-\beta_i)
 =\sum_{y\in Y_i}h_y-\card{Y_i}\beta_i=0.
\]
Hence
\Href{Lemma}{lem:affine_conic_identity} gives fixed vectors
$v_1,v_2$ such that, for every generic $u$ and $k \in \brackets{2}$,
\[
 G_{h,k}(u) = G_{\beta,k}(u)+\iprod{v_k}{u}.
\]
Define
\[
 \overline L_\beta(u) :=
 \frac{G_{\beta,1}(u)}{S_1}+\frac{G_{\beta,2}(-u)}{S_2},
 \qquad v := \frac{v_1}{S_1}-\frac{v_2}{S_2}.
\]
The finite average \eqref{eq:mean_strip_width} now satisfies
\[
 \overline L(u) = \overline L_\beta(u)+\iprod{v}{u},
 \qquad 0 \le \overline L_\beta(u) \le K_0.
\]
The vector $v$ is independent of $u$.
This is the useful consequence of the affine conic identity:
all height differences with zero color averages contribute a
single linear function, even though the covering bases vary
with the direction.

Take $u_0$ from \Href{Lemma}{lem:separated_direction} and choose
$u \in \braces{u_0,-u_0}$ with $\iprod{v}{u} \le 0$.
Both orientations satisfy the same separation bound, while
this choice of sign gives
\[
 0\le\overline L(u)\le\overline L_\beta(u)\le K_0.
\]
The first inequality uses non-emptiness of every rainbow
intersection. Geometrically, we have found a direction for
which the containing strips have small determinant-weighted
average width.

It remains to find one pair with both a small strip width and
a sufficiently large basis determinant. We use the same
weights $w_u$ for both requirements. Write
$\mathcal Q=(B_1,B_2)\in\mathcal P(u)$ for a pair, and put
\[
 t := \frac{D_0}{4N^d}.
\]
The pairs to be excluded are
\[
\begin{aligned}
 \mathcal P_L &:= \braces{\mathcal Q \in \mathcal P(u):L_u(\mathcal Q)>2K_0},\\
 \mathcal P_D &:= \braces{\mathcal Q \in \mathcal P(u):\abs{\det B_1}<t}.
\end{aligned}
\]
Because every strip width is nonnegative, their average bounds
the contribution of the first excluded family:
\[
 2K_0\sum_{\mathcal Q\in\mathcal P_L}w_u(\mathcal Q)
 \le\sum_{\mathcal Q\in\mathcal P_L}
                  w_u(\mathcal Q)L_u(\mathcal Q)
 \le\overline L(u)\le K_0.
\]
Thus, $\mathcal P_L$ carries at most half of the total weight.
For $\mathcal P_D$, sum first over the second basis. The product
form of $w_u$ removes that sum and gives
\[
 \sum_{\mathcal Q \in \mathcal P_D}w_u(\mathcal Q)
 = \frac{1}{S_1}
 \sum_{\substack{B_1 \in \mathcal B_1(u)\\\abs{\det B_1}<t}}
 \abs{\det B_1}
 \le \frac{N^dt}{D_0} = \frac{1}{4}.
\]
Here we used the bound $N^d$ on the number of first-group bases
and $S_1\ge D_0$. The weight of the pairs outside both families
is therefore at least $1-1/2-1/4=1/4$. Choose one such pair.
For its normals and conic coefficients, we have
\[
 \sum_{i\in\brackets{2d}}a_i r_i=0,
 \qquad 0\le L\le2K_0,
 \qquad \abs{\det B_1}\ge t.
\]
Moreover, separation and \eqref{eq:coefficient_bounds} give
$a_i\ge\eta/\sqrt N$ for every $i\in\brackets{2d}$.

The determinant of $B_1$ controls one set of $d$ independent
normals. The coefficient lower bounds ensure that these normals
remain quantitatively independent after multiplication by the
$a_i$. To express this in terms of the operator needed by the
dual ellipsoid lemma, form
$U := [a_1r_1\ \dots\ a_{2d}r_{2d}]$.
Then $M = U\transpose{U}$. For a subset $J \subset \brackets{2d}$
of size $d$, let $R_J$ be the matrix of the corresponding normals,
ordered by their color indices. The Cauchy--Binet formula gives
\[
 \det M =
 \sum_{\substack{J \subset \brackets{2d}\\\card{J} = d}}
 \parenth{\prod_{i \in J}a_i^2}\abs{\det R_J}^2
 \ge \parenth{\prod_{i \in I_1}a_i^2}\abs{\det B_1}^2
 \ge \parenth{\frac{\eta^2}{N}}^d t^2.
\]
Taking square roots proves \eqref{eq:joint_determinant}.
Only the term indexed by $I_1$ is needed in this sum of
nonnegative squared determinants. This explains why the
selection argument only had to bound the determinant of the
first basis from below.
\end{proof}

The two representations can also be recorded by the affine lift
\[
 \widetilde Y_i := 
 \braces{(y,0,h_y),(0,y,h_y):y \in Y_i} \subset \R^{2d+1}.
\]
Their weighted sum is $(u,-u,L)$. The first $2d$ coordinates
retain the conic representations, while the last one records
the height controlled in \Href{Lemma}{lem:joint_selection}.

\subsection{Completion of the Helly proof}

\begin{proof}[Proof of \Href{Theorem}{thm:reduced}]
If some rainbow intersection is empty, there is nothing to prove.
Otherwise apply \Href{Definition}{dfn:reduced_construction}
to the polytopes after discarding redundant inequalities, with
$N = 2d$. By \Href{Lemma}{lem:normalized_halfspaces}, the resulting
colors are normalized, their normals remain in general position,
and their average heights are bounded by
$b := \kappa_d^{-1/d}$.

Apply \Href{Lemma}{lem:joint_selection} and then
\Href{Lemma}{lem:dual_ellipsoid} to the selected halfspaces.
Their intersection $P$, in the volume-preserving coordinates,
satisfies, by \eqref{eq:dual_ellipsoid_volume} and
\eqref{eq:joint_determinant},
\[
 \vol{d}P \le 
 \kappa_d\parenth{\frac{2\sqrt{2}K_0\sqrt N}{\eta}}^d
 \frac{4N^d}{D_0}.
\]
Substituting \eqref{eq:basis_scale} and
\eqref{eq:height_scale}, and using $\kappa_d b^d = 1$, gives
the explicit choice
\[
 A_d := 
 4(4\sqrt{2}d)^d\delta_{2d}^{-(d+1)}
 \eta_{d,2d}^{-d}(2d)^{2d^2+3d/2}.
\]
The bounds in \eqref{eq:delta} and \eqref{eq:separation_scale}
give $A_d = d^{O(d^2)}$. The coordinate change preserves volume, proving
the reduced theorem.
\end{proof}

Finally, \Href{Lemma}{lem:reduction} gives
\Href{Theorem}{thm:colorful_volume_helly} with
$\gamma_d = (Q_d A_d)^{-1} = d^{-O(d^2)}$.

\section*{Acknowledgments}

The author thanks M\'arton Nasz\'odi, J\'anos Pach, and Imre
B\'ar\'any for their infinite support and helpful discussions.

\bibliographystyle{alpha}
\begingroup
\raggedright
\bibliography{../work_current/uvolit}
\endgroup

\end{document}